\documentclass[review]{elsarticle}

\usepackage{dsfont}
\usepackage{amsmath}
\usepackage{amsfonts, amssymb, mathrsfs}
\usepackage{color}
\usepackage[colorlinks=true,linkcolor=blue]{hyperref}

\usepackage{lineno}

\modulolinenumbers[5]

\journal{Journal of Approximation Theory}

\newtheorem{theorem}{Theorem}[section]
\newtheorem{lemma}[theorem]{Lemma}

\newtheorem{definition}{Definition}[section]
\newproof{proof}{Proof}

\newcommand{\Cn}{ \mathds{C}^n }
\newcommand{\CO} {\Cn\setminus\Omega}
\newcommand{\COO} {\Cn\setminus\overline\Omega}
\newcommand{\Om}{\Omega = \left\{z\in\Cn : \rho(z)<0 \right\}}
\newcommand{\dom}{ {\partial\Omega} }
\newcommand{\OO}{\Omega_\eps\setminus\Omega}

\renewcommand{\O}{\Omega}
\newcommand{\eps}{\varepsilon}

\newcommand{\C}{\mathds{C}}
\newcommand{\R}{\mathds{R}}
\newcommand{\N}{\mathds{N}}

\newcommand{\RE}{\text{\rm{Re}}}
\newcommand{\IM}{\text{\rm{Im}}}
\newcommand{\pr}{ \text{\rm{pr}}}

\newcommand{\f}{{\mathbf{f}}}

\newcommand{\SK}{\sum\limits_{k=1}^\infty}

\newcommand{\cn}{\frac{1}{(2\pi i)^n}}
\newcommand{\dbar}{\bar\partial}
\newcommand{\intl}{\int\limits}
\newcommand{\suml}{\sum\limits}

\newcommand{\dS}[1]{\partial{\rho}(#1)\wedge(\bar\partial\partial{\rho}(#1))^{n-1}}

\newcommand{\dist}[2]{ \text{\rm{dist}} (#1,\ #2)}

\newcommand{\supp}[1]{ \text{\rm{supp}}\ {#1} }
\newcommand{\BMO}{ \text{\rm{BMO}}}

\newcommand{\scp}[2]{ \left\langle #1,\ #2  \right\rangle }
\newcommand{\V}[2]{\scp{\partial\rho(#1)}{#1 - #2}}
\newcommand{\norm}[1]{\left\lVert #1 \right\rVert}
\newcommand{\abs}[1]{\left\lvert #1 \right\rvert}

\begin{document}

\begin{frontmatter}

\title{Constructive description of Hardy-Sobolev spaces in $\Cn$}

\author[mymainaddress]{Alexander Rotkevich}
\ead{rotkevichas@gmail.com}
\address[mymainaddress]{Department of Mathematical analysis, Mathematics and Mechanics Faculty, St. Petersburg State University, 198504,  Universitetsky prospekt, 28, Peterhof, St. Petersburg, Russia}

\begin{abstract}
In this paper we study the polynomial approximations in
Hardy-Sobolev spaces on for convex domains. We use the method of
pseudoanalytical continuation to obtain the characterization of
these spaces in terms of polynomial approximations.
\end{abstract}

\begin{keyword} Hardy-Sobolev Spaces \sep polynomial approximations,
pseudoanalytical continuation \sep Cauchy-Leray-Fantappi\`{e}
integral
\MSC[2010] 32E30\sep 41A10
\end{keyword}

\end{frontmatter}

\section{Introduction\label{intro}}

The purpose of this paper is to give an alternative
characterizations of Hardy-Sobolev (see. \cite{AB88}) spaces
\begin{equation} \label{def:Hardy-Sobolev}
 H^l_p(\O) =\{f\in H(\O): \norm{f}_{H^p(\O)} + \suml_{\abs{\alpha}\leq l}\norm{\partial^{\alpha} f}_{H^p(\O)}<\infty\}
\end{equation} on strongly convex domain $\O\subset\Cn.$

We continue the research started in \cite{R13} and devoted to
description of basic spaces of holomorphic functions of several
variables in terms of polynomial approximations and pseudoanalytical
continuation. In particular, we show that for $1<p<\infty$ and
$l\geq 1$ a holomorphic on a strongly convex domain $\O$ function
$f$ is in the Hardy-Sobolev space $H^l_p(\O)$ if and only if there
exist a sequence of $2^k-$degree polynomials $P_{2^k}$ such that
\begin{equation}
 \intl_\dom d\sigma(z) \left(\SK \abs{f(z)-P_{2^k}(z)}^2 2^{2l k}
 \right)^{p/2} < \infty. \label{eq:SobSp_cond}
\end{equation}
In the one variable case this condition follows from the
characterization obtained by E.M. Dynkin \cite{D81} for Radon
domains.

The paper is divided into five sections with one appendix. In
section~\ref{notations} we give main definitions and preliminaries
of this work. Section~\ref{CLF} is devoted to the
Cauchy-Leray-Fantappi\`{e} integral formula, the polynomial
approximations and estimates of its kernel. We also define internal
and external Kor\'{a}nyi regions, the multidimensional analog of
Lusin regions. In section~\ref{PAC} we introduce the method of
pseudoanalytical continuation and three constructions of the
continuation with different estimates. We use these constructions to
obtain the characterization of Hardy-Sobolev spaces in terms of
estimates of the pseudoanalytical continuation. To prove this result
we use the special analog of the Krantz-Li area-integral
inequality~\cite{KL97} for external Kor\'{a}nyi regions established
in appendix~\ref{Area_int}. Finally, section~\ref{Poly_Sobolev}
contains the proof of characteristics~(\ref{eq:SobSp_cond}).

\section{Main notations and definitions \label{notations}}

Let $\Cn$ be the space of $n$ complex variables, $n\geq 2,$ $z =
(z_1,\ldots, z_n),\ z_j = x_j + i y_j;$
$$\partial_j f =\frac{\partial f}{\partial z_j} = \frac{1}{2}\left( \frac{\partial f}{\partial x_j} - i \frac{\partial f}{\partial y_j}\right), \quad \bar\partial_j f = \frac{\partial f}{\partial\bar{z}_j} = \frac{1}{2}\left( \frac{\partial f}{\partial x_j} + i \frac{\partial f}{\partial y_j}\right),$$

$$\partial f = \suml_{k=1}^{n} \frac{\partial f}{\partial z_k} dz_k,\quad \bar{\partial} f = \suml_{k=1}^{n} \frac{\partial f}{\partial\bar{z}_k} d\bar{z}_k,\quad df=\partial f+  \bar{\partial} f.$$

\noindent The notation
$$\scp{\partial f(z)}{w} = \suml_{k=1}^{n} \frac{\partial f(z)}{\partial z_k} w_k.$$
is used to indicate the action of $\partial f$ on the vector
$w\in\Cn,$ and $$|\bar\partial f| = \abs{\frac{\partial f}{\partial
z_1}} + \ldots+\abs{\frac{\partial f}{\partial z_n}}.$$

The euclidean distance form the point $z\in\Cn$ to the set
$D\subset\Cn$ we denote as $\dist{z}{D} = \inf\{\abs{z-w}:w\in D\}.$
Lebesgue measure in $\Cn$ we denote as $d\mu.$

For a multiindex $\alpha =
(\alpha_1,\ldots,\alpha_n)\in\mathds{N_0}^n$ we set
$\abs{\alpha}~=~\alpha_1~+~\ldots~+~\alpha_n$ and
$\alpha!~=\alpha_1!\ldots\alpha_2!,$
 also $z^\alpha =
z_1^{\alpha_1}\ldots z_n^{\alpha_n}$ and $\partial^\alpha f =
\frac{\partial^{\abs{\alpha}} f}{\partial
\bar{z}_1^{\alpha_1}\ldots\partial \bar{z}_n^{\alpha_n}}.$



Let $\Om$ be a strongly convex domain with a $C^3$-smooth defining
function. We need to consider a family of domains $$\Omega_t =
\left\{z\in\Cn : \rho(z)<t \right\}$$ that are also strongly convex
for each $|t|<\eps,$ where $\eps>0$ is small enough, that is $d^2\rho(z)$ is positive definite when $|\rho(z)|\leq\eps.$ For $z\in \O_{\eps}\setminus\O_{-\eps}$ we denote the nearest point on $\dom$ as $\pr_{\dom}(z).$ Then the mapping 
$$\pr_{\dom} : \O_{\eps}\setminus\O_{-\eps}\to \dom$$
is well defined, $C^2-$smooth on $\O_{\eps}\setminus\O$ and $|z-\pr_{\dom}(z)|=\dist{z}{\dom}.$

For $\xi\in\dom_t$ we define the complex tangent space
$$T_\xi = \left\{ z\in\Cn : \scp{\partial{\rho}(\xi)}{\xi-z} = 0  \right\}.  $$

The space of holomorphic functions we denote as $H(\O)$ and consider
the Hardy space (see~\cite{S76},~\cite{FS72}) $$H^p(\O):=\left\{f\in
H(\O):\ \norm{f}_{H^p(\O)}^p=\sup\limits_{-\eps<t<0} \intl_{\dom_t} |f(z)|^p d\sigma_t
(z) <\infty\right\},$$ where $d\sigma_t$ is induced Lebesgue measure
on the boundary of $\O_t.$ We also denote $d\sigma = d\sigma_0.$
Hardy-Sobolev spaces $H^l_p(\O)$ are defined
by~(\ref{def:Hardy-Sobolev}).

Throughout this paper we use notations $\lesssim,\ \asymp.$  We let
$f\lesssim g$ if $f\leq c g$ for some constant $c>0,$ that doesn't
depend on main arguments of functions $f$ and $g$ and usually depend
only on dimension $n$ and domain $\O.$ Also $f\asymp g$ if $c^{-1}
g\leq f\leq c g$ for some $c>1.$

\section{Cauchy-Leray-Fantappi\`{e} formula \label{CLF}}

In the context of theory of several complex variables there is no
unique reproducing formula formula, however we could use the Leray
theorem, that allows us to construct holomorphic reproducing kernels
(\cite{AYu79}, \cite{L59}, \cite{Ra86}). For convex domain $\Om$
this theorem brings us Cauchy-Leray-Fantappi\`{e} formula, and for
$f\in H^1(\Omega)$ and $z\in\O$ we have
\begin{equation} \label{eq:CLF}
 f(z) = K_\O f(z) = \cn \int\limits_{\dom} \frac{f(\xi) \partial\rho(\xi)\wedge(\bar{\partial}\partial\rho(\xi))^{n-1}}{\scp{\partial\rho(\xi)}{\xi-z}^n} = \int\limits_{\dom} f(\xi) K(\xi,z) \omega(\xi),
\end{equation}
where $\omega(\xi) = \cn
\partial\rho(\xi)\wedge(\bar{\partial}\partial\rho(\xi))^{n-1},$ and
$K(\xi,z) =\scp{\partial\rho(\xi)}{\xi-z}^{-n}.$

The $(2n-1)$-form $\omega$ defines on $\dom_t$ Leray-Levy measure
$dS,$ that is equivalent to Lebesgue surface measure $d\sigma_t$
(for details see \cite{AYu79}, \cite{LS13}, \cite{LS14}). This
allows us to identify Lebesgue, Hardy and Hardy-Sobolev spaces
defined with respect to measures $d\sigma_t$ and $dS$. Also note,
that measure $dV$ defined by the $2n$-form
$d\omega=(\partial\dbar\rho)^{n}$ is equivalent to Lebesgue measure
$d\mu$ in $\Cn.$

By \cite{R12} the integral operator $K_\O$ defines a bounded mapping
on $L^p(\dom)$ to $H^p(\Omega)$ for $1<p<\infty.$

The function $d(w,z) = \abs{\V{w}{z}}$  defines on $\dom$
quasimetric, and if $B(z,\delta) = \{w\in\dom: d(w,z)<\delta\}$ is a
quasiball with respect to $d$ then $\sigma(B(z,\delta))\asymp
\delta^n,$ see for example \cite{R12}. Therefore $\{\dom,d,\sigma\}$
is a space of homogeneous type.

Note also the crucial role in the forthcoming considerations of the
following estimate that is proved in \cite{R13}.
\begin{lemma} \label{lm:QM_est1} Let $\O$ be strongly convex, then
$$d(w,z) \asymp \rho(w) + d(\pr_{\dom}(w),z),\ w\in\CO,\ z\in\dom.$$
\end{lemma}

\subsection{The polynomial approximation of Cauchy-Leray-Fantappi\'{e} kernel\label{CLF_approx}}
In lemma~\ref{lm:CLF_approx} here we construct a polynomial
approximations of Cauchy-Leray-Fantappi\'{e} kernel based on theorem
by V.K. Dzyadyk about estimates of Cauchy kernel on domains on
complex plane (theorem~1 in part~1 of section~7 in~\cite{Dz77}). The
approximation is choosed similarly to \cite{Sh89}.
This construction allows us in theorem~\ref{thm:Poly_Sobolev} to get
polynomials that approximate holomorphic function with desired
speed.

\begin{lemma}
Let $\O$ be a strongly convex domain with $0\in\O,$ then for every
$\xi\in\OO$ the value of $\lambda =
\frac{\scp{\partial\rho(\xi)}{z}}{\scp{\partial\rho(\xi)}{\xi}}$ for
$z\in\O$ lies in domain $L(t),$ bounded by the bigger arc of the
circle $\abs{\lambda}=R=R(\O)$ and the chord $\{\lambda\in\C :
\lambda = 1 + e^{it}s,\ s\in\R,\ \abs{\lambda}\leq R \},$ where $t =
\frac{\pi}{2} - \arg(\scp{\partial\rho(\xi)}{\xi}).$
\end{lemma}
\begin{proof}
For $\xi\in\dom$ define
 $$\Lambda(\xi) = \left\{\lambda\in\C : \lambda =  \frac{\scp{\partial\rho(\xi)}{z}}{\scp{\partial\rho(\xi)}{\xi}},\ z\in\O\right\}. $$
The convexity of $\O$ with $0\in\O$ implies that
\begin{equation} \label{eq1}
 \abs{\scp{\partial\rho(\xi)}{\xi}} \gtrsim \abs{\partial\rho(\xi)} \abs{\xi} \gtrsim 1,
\end{equation}
\begin{equation} \label{eq2}
 \RE \scp{\partial\rho(\xi)}{z-\xi} \leq 0,\quad z\in\bar{\O},\ \xi\in\OO.
\end{equation}
The domain $\Lambda(\xi)\subset\C$ is also convex and contains 0, thus the
equality
$$ \frac{\scp{\partial\rho(\xi)}{z}}{\scp{\partial\rho(\xi)}{\xi}} = 1 + \frac{\scp{\partial\rho(\xi)}{z-\xi}}{\scp{\partial\rho(\xi)}{\xi}} $$
with estimates (\ref{eq1}), (\ref{eq2}) completes the proof of the
lemma. \qed
\end{proof}

\begin{lemma} \label{lm:CLF_approx}
Let $\O$ be a strongly convex domain and $r>0.$ Then for every
$k\in\N$ there exist function $K^{glob}_k(\xi,z)$ defined for $\xi\in\OO$
and polynomial in $z\in\O$ with $\deg K_k(\xi,\cdot)\leq k$ and
following properties:

\begin{equation} \label{eq:CLF_approx1}
 \abs{K(\xi,z) - K^{glob}_k(\xi,z)} \lesssim \frac{1}{k^{r}} \frac{1}{d(\xi,z)^{n+r}},\quad d(\xi,z)\geq \frac{1}{k};
\end{equation}

\begin{equation} \label{eq:CLF_approx2}
 \abs{K^{glob}_k(\xi,z)}\lesssim k^n,\quad d(\xi,z)\leq\frac{1}{k}.
\end{equation}

\end{lemma}

\begin{proof}
Due to \cite{Dz77} and \cite{Sh93} for any $j\in\N$ there exists a
function $T_j(t,\lambda)$ polynomial in $\lambda$ with $\deg
T_j(t,\cdot)\leq j$ such that
\begin{equation} \label{eq:Cauchy_approx1}
\left\vert \frac{1}{1-\lambda} - T_j(t,\lambda) \right\vert \lesssim
\frac{1}{j^r}\frac{1}{\abs{1-\lambda}^{1+r}}
\end{equation} for $\lambda\in
L(t)\setminus\left\{\lambda : \abs{1-\lambda}<\frac{1}{j}\right\}$
and coefficients of polynomials $T_j(t,\lambda)$ continuously depend
on $t.$ Note also that by maximum principle
\begin{equation} \label{eq:Cauchy_approx2}
T_j(t,\lambda) \lesssim j,\quad \lambda\in
L(t)\bigcap\left\{\lambda : \abs{1-\lambda}<\frac{1}{j}\right\}.
\end{equation}

Let $t(\xi) = \frac{\pi}{2} - \arg(\scp{\partial\rho(\xi)}{\xi}$ and for $j\in\N$ and $(j-1)<k\leq jn$ define $$ K_k^{glob}(\xi,z) =
K_{jn}^{glob}(\xi,z) = \frac{1}{\scp{\partial\rho(\xi)}{\xi}^n}
T_j^n\left(t(\xi),\frac{\scp{\partial\rho(\xi)}{z}}{\scp{\partial\rho(\xi)}{\xi}}\right).$$
Due to definition of $T_j$ polynomials $K_{k}^{glob}(\xi,\cdot)$
satisfy relations~(\ref{eq:CLF_approx1}), (\ref{eq:CLF_approx2}). \qed

\end{proof}

\subsection{Kor\'{a}nyi regions}
For $\xi\in\dom$ and $\eps>0$ we define the {\it inner Kor\'{a}nyi
region} as
$$D^i(\xi,\eta,\eps) = \{\tau\in\O : \pr_{\dom}(\tau) \in B(\xi,-\eta\rho(\tau)),\ \rho(\tau)>-\eps \}. $$

The strong convexity of $\O$ implies that area-integral inequality
by S.~Krantz and S.Y.~Li~\cite{KL97} for $f\in H^p(\O),\
0<p<\infty,$ could be expressed as

\begin{equation} \label{ineq:Luzin_internal}
  \intl_\dom d\sigma(z) \left( \intl_{D^i(z,\eta,\eps)} \abs{\partial f(\tau)}^2
  \frac{d\mu(\tau)}{(-\rho(\tau))^{n-1}}\right)^{p/2} \leq c(\O,p) \intl_\dom \abs{f}^p d\sigma.
\end{equation}

Consider the decomposition of vector $\tau\in\Cn$ as $\tau = w +
t n(\xi),$ where $w\in T_\xi,\ t\in\C,$ and
$n(\xi)=\frac{\bar\partial\rho(\xi)}{\abs{\bar\partial\rho(\xi)}}$ is a
complex normal vector at $\xi$. We define the {\it external
Kor\'{a}nyi region} as
\begin{multline}\label{df:KoranyExt}
 D^e(\xi,\eta,\eps) = \{\tau\in\CO : \tau = w + t n(\xi),\\ w\in T_\xi,\
t\in\C,\ \abs{w}<\sqrt{\eta\rho(\tau)},\ \abs{\IM(t)}<\eta\rho(\tau),\
\rho(\tau)<\eps \}.
\end{multline}
In appendix~\ref{Area_int} we will proof the area-integral
inequality similar to~(\ref{ineq:Luzin_internal}) for external
regions~$D^e(\xi,\eta,\eps).$

We point out two rules for integration over regions $D^e(\xi,\eta,\eps).$
First, for every function $F$ we have
$$\intl_{\O_\eps\setminus\O} \abs{F(z)} d\mu(z) \asymp \intl_\dom d\sigma(\xi)
\intl_{D^e(\xi,\eta,\eps)} \abs{F(\tau)}
\frac{d\mu(\tau)}{\rho(\tau)^{n}}.
$$
Second, if $F(w) = \tilde{F}(\rho(w))$ then
$$ \intl_{D^e(\xi,\eta,\eps)} \abs{F(\tau)} d\mu(\tau) \asymp \intl_0^\eps \abs{\tilde{F}(t)} t^{n} dt.  $$
Similar rules are valid for regions $D^i(\xi,\eta,\eps).$

We could clarify the estimate of $d(\tau,w)$ in
lemma~\ref{lm:QM_est1} for $\tau\in D^e(z,\eta,\eps).$

\begin{lemma} \label{lm:QM_est2} Let $\O$ be a strongly convex domain and $\eps,\eta>0,$
then
\begin{equation} \label{ineq:QM_est2}
 d(\tau,w) \asymp \rho(\tau) + d(z,w),\quad z,w\in\dom,\ \tau\in
 D^e(z,\eta,\eps).
\end{equation}
\end{lemma}

\begin{proof}
For $\tau\in D^e(z,\eta,\eps)$ we denote $\hat{\tau}=\pr_{\dom}(\tau),$ then $d(\hat{\tau},z)\lesssim
\eta\rho(\tau)$ and by lemma~\ref{lm:QM_est1}
\begin{equation*}
d(\tau,w)\lesssim \rho(\tau)+d(\hat{\tau}\lesssim
\rho(\tau)+d(\hat{\tau},z) +d(z,w) \lesssim \rho(\tau) +d(z,w).
\end{equation*}
On the other hand,
\begin{multline*} \rho(\tau) +d(z,w) \lesssim \rho(\tau) +
(d(z,\hat{\tau})+d(\hat{\tau},w))\lesssim
(1+\eta)\rho(\tau)+d(\hat{\tau},w)\\
\lesssim \rho(\tau)+d(\hat{\tau},w) \lesssim d(\tau,w). 
\end{multline*}
\qed
\end{proof}

\section{The method of pseudoanalytical continuation \label{PAC}}

\subsection{Definition of pseudoanalytical continuation}
The main tool of this paper is the method of continuation of
function $f\in H(\O)$ outside the domain $\Omega.$ Let $f\in
H^1(\Omega)$ and let the boundary values of $f$ almost everywhere
coincide with the boundary values of some function $\f\in
C_{loc}^1(\COO)$ such that $ \abs{\bar{\partial}\f}\in L^1(\CO).$
Then by Stokes formula for $z\in\O$ we have
\begin{multline*} 
f(z) = \lim\limits_{r\to 0+ } \cn \int\limits_{\dom_r} \frac{\f(\xi) \partial\rho(\xi)\wedge(\bar{\partial}\partial\rho(\xi))^{n-1}}{\scp{\partial\rho(\xi)}{\xi-z}^n} = \\
\lim\limits_{r\to 0+ } \cn \int\limits_{\CO_r}
\frac{\bar{\partial}\f(\xi)\wedge
\partial\rho(\xi)\wedge(\bar{\partial}\partial\rho(\xi))^{n-1}}{\scp{\partial\rho(\xi)}{\xi-z}^n}\\
=  \cn \int\limits_{\CO} \frac{\bar{\partial}\f(\xi) \wedge
\partial\rho(\xi)\wedge(\bar{\partial}\partial\rho(\xi))^{n-1}}{\scp{\partial\rho(\xi)}{\xi-z}^n},
\end{multline*}
since (for details see \cite{Ra86})
$$ d_\xi\left( \frac{\partial\rho(\xi)\wedge(\bar{\partial}\partial\rho(\xi))^{n-1}}{\scp{\partial\rho(\xi)}{\xi-z}^n} \right) = 0,\quad z\in\O,\ \xi\in\CO. $$

This formula allows us to study properties of function $f\in H(\O)$
relying on estimates of its continuation.
\begin{definition} We call the function $\f\in
C_{loc}^1(\Cn\setminus\overline{\O})$ {\it the pseudoanalytic continuation of the function}  $f\in H(\O)$ if
\begin{equation} \label{eq:PAC}
f(z) = \cn \int\limits_{\CO} \frac{\bar{\partial}\f(\xi) \wedge
\partial\rho(\xi)\wedge(\bar{\partial}\partial\rho(\xi))^{n-1}}{\scp{\partial\rho(\xi)}{\xi-z}^n},\ z\in\O.
\end{equation}
\end{definition}
Note that it is not
necessary for the function~$\f$ to be a continuation in terms of
coincidence of boundary values. 

\subsection{Continuation by symmetry}

For $z\in\O_\eps\setminus\O$ we define the symmetric along $\dom$ point $z^*\in\O$ by $$z^* - z = 2(\pr_{\dom}(z)-z).$$

\begin{theorem} \label{thm:PAC_sym}
 Let $f\in H^1_p(\O)$ and $1<p<\infty,\ m\in\N.$ There exist a pseudoanalytical continuation
 $\f\in C_{loc}^1(\COO)$ of function $f$ such that
 $\supp{\f}\subset\O_\eps,$
 $\abs{\dbar \f(z)}\in L^p(\O_\eps\setminus\overline{\O})$ and
 \begin{equation} \label{ineq:PAC_sym1}
  \abs{\dbar \f(z)} \lesssim \max\limits_{\abs{\alpha}=m}
  \abs{\partial^{\alpha}f(z^*)}
  \rho(z)^{m-1},\quad z\in\O_\eps\setminus\O.
 \end{equation}
\end{theorem}

\begin{proof}
Define
\begin{equation} \label{ineq:PAC_sym0}
 \f_0(z) = \sum\limits_{\abs{\alpha}\leq m-1} \partial^{\alpha}f(z^*)
 \frac{(z-z^*)^\alpha}{\alpha !},\ z\in\OO.
\end{equation}
Let $\alpha\pm e_k=(\alpha_1,\ldots,\alpha_k\pm 1,\alpha_n)$ and
define $(z-z^*)^{\alpha-e_k}=0$ if $\alpha_k=0.$ In these notations
we have
\begin{multline}
 \dbar_j \f_0=\SK \sum\limits_{\abs{\alpha}\leq m-1} \left(\partial^{\alpha+e_k}f(z^*)
 \frac{(z-z^*)^\alpha}{\alpha !} - \partial^{\alpha}f(z^*)
 \frac{(z-z^*)^{\alpha-e_k}}{(\alpha-e_k)!}\right)\dbar_j z^*_k\\
 = \SK \sum\limits_{\abs{\alpha}=m-1} \partial^{\alpha+e_k}f(z^*)
 \frac{(z-z^*)^\alpha}{\alpha !} \dbar_j z^*_k,
\end{multline}
hence, $$ \abs{\dbar \f_0(z)} \lesssim \max\limits_{\abs{\alpha}=m}
\abs{\partial^{\alpha}f(z^*)}
  \rho(z)^{m-1},\quad z\in\CO. $$
Consider function $\chi\in C^\infty(0,\infty)$ such that $\chi(t)=1$
for $t\leq \eps/2$ and $\chi(t)=0$ for $t\geq \eps.$ The function
$\f(z) = \f_0(z)\chi(\rho(z))$ satisfies the condition
(\ref{ineq:PAC_sym0}) and $\supp\f~\subset~\O_\eps.$

Let $d = \dist{z^*}{\dom}/10,$ then for every mutiindex $\alpha$
such that $\abs{\alpha}=m$ by Cauchy maximal inequality we have
$$\abs{\partial^\alpha f(z^*)} \lesssim d^{-m+1} \sup\limits_{\abs{\tau-z^*}<d}
\abs{\partial f(\tau)} \lesssim \rho(z)^{-m+1} \sup\limits_{\tau\in
D^i(\pr_{\dom}(z),c_0d,\eps)} \abs{\partial f(\tau)},  $$ for some $c_0>0.$
Finally, by theorem~2.1 from \cite{KL97} we get
\begin{multline*}
\int_{\OO} \abs{\dbar \f(z)}^p d\mu(z) \lesssim \int_{\O\setminus\O_{-\eps}} d\mu(z) \left(\sup\limits_{\tau\in
D^i(\pr_{\dom}(z),c_0d,\eps)} \abs{\partial f(\tau)}\right)^p\\
\lesssim \norm{\partial f}^p_{H^p(\O)}<\infty.
\end{multline*}
\qed
\end{proof}

\subsection{Continuation by global approximations.\label{PAC_glob}}
Let $f\in H^1(\Omega)$ and consider a polynomial sequence
$P_1,P_2,\ldots$ converging to $f$ in $L^1(\dom).$ Define
$$ \lambda(z) = \rho(z)^{-1} \abs{P_{2^{k+1}}(z) - P_{2^{k}}(z)},\quad 2^{-k}<\rho(z)\leq 2^{-k+1}.$$

\begin{theorem} \label{thm:PAC_glob}
 Assume that $\lambda\in L^p(\Cn\setminus\Omega)$ for some $p\geq 1.$ Then there exist a pseudoanalytical continuation $\f$ of function $f$
 such that
\begin{equation}\label{eq:PAC_la}
\abs{\bar{\partial} \mathbf{f}(z)} \lesssim \lambda(z),\quad
z\in\Cn\setminus\Omega.
\end{equation}
\end{theorem}

\begin{proof}
Consider function $\chi\in C^\infty(0,\infty)$ such that $\chi(t)=1$
for $t\leq \frac{5}{4}$ and $\chi(t)=0$ for $t\geq \frac{7}{4}.$ We let 
$$  \mathbf{f}_0(z) =  P_{2^{k}}(z) + \chi(2^k\rho(z)) (P_{2^{k+1}}(z) - P_{2^{k}}(z)), \quad 2^{-k}<\rho(z)<2^{-k+1},\ k\in\mathbb{N},$$
and define the continuation of a function $f$ by formula $\f=\chi(2\rho(z))\f_0(z).$

Now $\mathbf{f}$ is $C^1$-function on $\Cn\setminus\overline{\Omega}$ and
$\abs{\bar{\partial} \mathbf{f}(z)} \lesssim \lambda(z).$ We define
a function $F_k(z)$ as $F_k(z)= \mathbf{f}(z)$ for $\rho(z)>2^{-k}$
and as $F_k(z) = P_{2^{k+1}}(z)$ for $\rho(z)<2^{-k}.$ The the
function $F_k$ is smooth and holomorphic in $\Omega_{2^{-k}},$ and
$\abs{\dbar F_k(z)}\lesssim\lambda(z)$ for
$z\in\Cn\setminus\Omega_{2^{-k}}.$ Thus similarly to \ref{eq:PAC} we
get
$$P_{2^{k+1}}(z) = F_k(z) = \cn \int\limits_{\CO} \frac{\bar{\partial}F_k(\xi) \wedge\dS{\xi} }{\V{\xi}{z}^n},\ z\in\Omega,$$
We can pass to the limit in this formula by the dominated
convergence theorem; hence, function~$\f$ satisfies the
formula~(\ref{eq:PAC}) and is a pseudoanalytical continuation of
function~$f$. \qed
\end{proof}

\subsection{Pseudoanalytical continuation of Hardy-Sobolev spaces\label{PAC_Sobolev}}
\begin{theorem} \label{thm:PAC_Sobolev}
Let $\O$ be a strongly convex domain, $1<p<\infty,$ $l\in\N$ and
$f\in H^p(\O).$ Then $f\in H_p^l(\O)$ if and only if there exists
such pseudoanalytical continuation $\f$ that for some $\eta>0$

\begin{equation} \label{ineq:PAC_Sobolev}
 \intl_\dom d\sigma(z) \left(\intl_{D^e(z,\eta,\eps)} \abs{\dbar \f(\tau) \rho(\tau)^{-l}}^2
 d\nu(\tau)
 \right)^{p/2} <\infty,
\end{equation}
where $d\nu(\tau)=\frac{d\mu(\tau)}{\rho(\tau)^{n-1}}.$
\end{theorem}

\begin{proof}
Let $f\in H^l_p(\O).$ By theorem~\ref{thm:PAC_sym} we could
construct pseudoanalytical continuation~$\f$ such that
$$\abs{\dbar \f(z)} \lesssim \max\limits_{\abs{\alpha}=l+1} \abs{\partial^\alpha
f(z^*)}
  \rho(z)^{l},\quad z\in\CO.$$

Note that the symmetry $(z\mapsto z^*)$ with respect to $\dom$ maps the external
sector $D^e(z,\eta,\eps)$ into some internal Kor\'{a}nyi sector. Indeed,
for every $\eta>0$ there exists $\eta_1,\eps_1>0$ such that
$$\{\tau^*:\tau\in D^e(z,\eta,\eps)\} \subseteq D^i(z,\eta_1,\eps_1).$$
Applying area-integral inequality~(\ref{ineq:Luzin_internal}) we
obtain
\begin{multline*}
\intl_\dom d\sigma(z) \left(\intl_{D^e(z,\eta,\eps)} \abs{\dbar \f(\tau)
\rho(\tau)^{-l}}^2
 d\nu(\tau)
 \right)^{p/2}\\
 \lesssim \max\limits_{\abs{\alpha}=l+1} \intl_\dom d\sigma(z) \left(\intl_{D^e(z,\eta,\eps)} \abs{\partial^{\alpha}f(\tau^*)}^2
 d\nu(\tau)
 \right)^{p/2}\\
 \lesssim \max\limits_{\abs{\alpha}=l+1} \intl_\dom d\sigma(z) \left(\intl_{D^i(z,\eta_1,\eps_1)} \abs{\partial^{\alpha}f(\tau) }^2
 \frac{d\mu(\tau)}{(-\rho(\tau))^{n-1}}
 \right)^{p/2} <\infty
\end{multline*}

To prove the sufficiency, assume that function $f\in H^1(\O)$ admits
the pseudoanalytical continuation~$\f$ with the estimate~(\ref{ineq:PAC_Sobolev}.) We will prove that for every function
$g\in L^{p'}(\dom),\ \frac{1}{p}+\frac{1}{p'} =1,$ and every
multiindex $\alpha,\ \abs{\alpha}\leq l,$

$$ \abs{ \intl_\dom g(z) \partial^{\alpha}f(z) dS(z)} \leq c(f)\norm{g }_{L^{p'}(\dom)}.$$

Assume, without loss of generality, that $\alpha=(l,0,\ldots,0).$ By
representation~(\ref{eq:PAC}) we have

$$ f(z) = \intl_{\CO} \frac{\bar{\partial}\f(\xi) \wedge \omega(\xi)}{\scp{\partial\rho(\xi)}{\xi-z}^n}$$
and with $C_{nl} = \frac{(n+l-1)!}{(n-1)!}$
\begin{multline*}
 \intl_\dom g(z)\partial^{\alpha}f(z)  dS(z)\\
 = C_{nl} \intl_\dom g(z) \left( \intl_{\CO} \left(\frac{\partial\rho(\xi)}{\partial\xi_1}\right)^{l} \frac{\bar{\partial}\f(\xi) \wedge \omega(\xi)}{\scp{\partial\rho(\xi)}{\xi-z}^{n+l}}\right)
 dS(z)\\
 = C_{nl} \intl_{\CO} \left(\frac{\partial\rho(\xi)}{\partial\xi_1}\right)^{l} \bar{\partial}\f(\xi) \wedge
 \omega(\xi) \intl_\dom
 \frac{g(z)dS(z)}{\scp{\partial\rho(\xi)}{\xi-z}^{n+l}}.
\end{multline*}
Define $\Phi_l(\xi) = \intl_\dom
 \frac{g(z)dS(z)}{\scp{\partial\rho(\xi)}{\xi-z}^{n+l}}. $
Applying H\"{o}lder inequality twice we have
\begin{multline*}
 \abs{\intl_\dom g(z) \partial^\alpha f(z) dS(z)} \lesssim
 \intl_{\CO} \abs{\dbar \f(\xi)} \abs{\Phi_l(\xi)} d\mu(\xi)\\
 \lesssim \intl_\dom dS(\xi) \intl_{D^e(\xi,\eta,\eps)} \abs{ \dbar \f(\tau) }
 \abs{\Phi_l(\tau)} \frac{d\mu(\tau)}{\rho(\tau)^n} \\
 \lesssim \intl_\dom d\sigma(\xi) \left(\intl_{D^e(\xi,\eta,\eps)} \abs{\dbar \f(\tau)
 }^2\rho(\tau)^{-2l} \frac{d\mu(\tau)}{\rho(\tau)^{n-1}}\right)^{1/2} \times\\
 \times\left(\intl_{D^e(\xi,\eta,\eps)} \abs{\Phi_l(\tau)}^2 \rho(\tau)^{2l-2}
 \frac{d\mu(\tau)}{\rho(\tau)^{n-1}}\right)^{1/2}\\
 \lesssim \left( \intl_\dom dS(\xi) \left(\intl_{D^e(\xi,\eta,\eps)} \abs{\dbar \f(\tau)}^2\rho(\tau)^{-2l}
 d\nu(\tau)\right)^{p/2}\right)^{1/p} \times\\ \times
 \left( \intl_\dom dS(\xi)  \left(\intl_{D^e(\xi,\eta,\eps)} \abs{\Phi_l(\tau)}^2 \rho(\tau)^{2l-2}
 d\nu(\tau)\right)^{p'/2}\right)^{1/p'}.
\end{multline*}
The first product term is bounded by~(\ref{ineq:PAC_Sobolev}), and
the second one by the area-integral inequality~(\ref{est:area_int}),
that we will prove in the appendix~\ref{Area_int} in
theorem~\ref{thm:area_int}. \qed
\end{proof}

\section{Constructive description of Hardy-Sobolev spaces \label{Poly_Sobolev}}

\begin{theorem} \label{thm:Poly_Sobolev} Let $f\in H^1(\O)$ and $1<p<\infty,\
l\in\N.$ Then  $f\in H^l_p(\O)$ iff there exists sequence of $2^k$-degree polynomials $P_{2^k}$ such that
\begin{equation} \label{ineq:Poly_Sobolev}
 \intl_{\dom} d\sigma(z) \left(\SK \abs{f(z)-P_{2^k}(z)}^2 2^{2l k}
 \right)^{p/2} < \infty.
\end{equation}
\end{theorem}

\begin{proof}
Assume that condition~(\ref{ineq:Poly_Sobolev}) holds, then
polynomials $P_{2^k}$ converge to function $f$ in $L^p(\dom)$ and by
the theorem~\ref{thm:PAC_glob} we could construct pseudoanalytical
continuation~$\f$ such that

$$ \abs{\dbar \f(z)} \lesssim \abs{P_{2^{k+1}}(z)-P_{2^k}(z)}
\rho(z)^{-1},\quad z\in \CO,\ 2^{-k}\leq \rho(z)<2^{-k+1}.$$

Consider the decomposition of region $D^e(z,\eta,\eps)$ to sets $D_k(z) =
\{\tau\in D^e(z,\eta,\eps):\ 2^{-k}\leq \rho(\tau) < 2^{-k+1}\},$ and
define functions
\begin{align*}
a_k(z) &= \abs{P_{2^{k+1}}(z)-P_{2^k}(z)} 2^{-kl },\\
b_k (z) &= \left(\intl_{D_k(z)} \abs{\dbar
\f(\tau)\rho(\tau)^{-l}}^2 d\nu(\tau) \right)^{1/2}, z\in\dom.
\end{align*}

\begin{lemma}\label{lm:Poly_Sobolev} $b_k(z) \lesssim M a_k (z),$
where $Ma_k$ is the maximal function with respect to centred quasiballs on $\dom$ $$Ma_k(z) = \sup\limits_{r>0}\frac{1}{\sigma(B(z,r))} \int\limits_{B(z,r)} |a_k(\xi)| d\sigma(\xi).$$
\end{lemma}

Assume, that this lemma holds, then by Fefferman-Stein maximal
theorem (see \cite{GLY04}, \cite{FS72}) we have

$$ \intl_\dom \left(\SK b_k(z)^2\right)^{p/2} d\sigma(z) \lesssim  \intl_\dom \left(\SK a_k(z)^2\right)^{p/2}d\sigma(z).$$
The right-hand side of this inequality is finite by the
condition~(\ref{ineq:Poly_Sobolev}), also we have
$$ \SK b_k(z)^2 = \intl_{D^e(z,\eta,\eps)} \abs{\dbar \f(\xi) \rho(\xi)^{-l}}^2 d\nu(\xi), $$
which completes the proof of the sufficiency in the theorem.

Prove the necessity. Now $f\in H^l_p(\O)$ with $1<p<\infty$ and
$l\in\N.$ By theorem~\ref{thm:PAC_Sobolev} we could construct
continuation~$\f$ of function~$f$ with
estimate~(\ref{ineq:PAC_Sobolev}). Applying the approximation of
Cauchy-Leray-Fantappi\`{e} kernel from lemma~\ref{lm:CLF_approx} to
function $\f$ we define polynomials
$$ P_{2^k}(z) = \int_{\CO}
\bar{\partial} \f(\xi)\wedge\omega(\xi) K^{glob}_{2^k}(\xi,z).$$ We
will prove that these polynomials satisfy the
condition~(\ref{ineq:Poly_Sobolev}). From lemma~\ref{lm:CLF_approx}
we obtain
\begin{multline*}
 \abs{f(z)-P_{2^k}(z)} \lesssim \int_{\CO} \abs{\bar{\partial}\f(\xi)} \abs{\frac{1}{\scp{\partial{\rho(\xi)}}{\xi-z}^d} - K^{glob}_{2^k}(\xi,z) } d\mu(\xi)\\ \lesssim U(z) + V(z) + W_1(z) +
 W_2(z),
\end{multline*}
where
\begin{align*}
U(z) &= \int\limits_{d(\tau,z)<2^{-k}} \frac{ \abs{\bar{\partial}\f(\tau)} }{ \abs{ \scp{\partial{\rho(\tau)}}{\tau-z} }^n }d\mu(\tau), \\
V(z) &= 2^{kn} \int\limits_{d(\tau,z)<2^{-k}} \abs{ \bar{\partial}\f(\tau) } d\mu(\tau), \\
W_1(z) &= 2^{-kr}  \int\limits_{\substack{d(\tau,z)>2^{-k}\\ \rho(\tau)<2^{-k}}} \frac{ \abs{\bar{\partial}\f(\tau)} }{ \abs{\scp{\partial{\rho(\tau)}}{\tau-z}}^{n+r} }d\mu(\tau),\\
W_2(z) &= 2^{-kr} \int\limits_{\rho(\tau)>2^{-k}} \frac{
\abs{\bar{\partial}\f(\tau)} }{
\abs{\scp{\partial{\rho(\tau)}}{\tau-z}}^{n+r} }d\mu(\tau).
\end{align*}
The parameter $r>0$ will be chosen later.

Note that $ V(z) \lesssim c U(z)$ and estimate the contribution of
 $U(z)$ to the sum. For some $c_1,c_2>0$ we have

\begin{multline*}
U(z) \leq \intl_{\substack{d(w,z)<c_1 2^{-k}\\ w\in\dom}} d\sigma(w)
\sum\limits_{j>c_2 k} \intl_{D_j(w)} \frac{ \abs{\dbar \f(\tau)} }{
\abs{\scp{\partial{\rho(\tau)}}{\tau-z}}^n }
\frac{d\nu(\tau)}{\rho(\tau)}\\
\leq \intl_{\substack{d(w,z)<c_1 2^{-k}\\ w\in\dom}} d\sigma(w)
\sum\limits_{j>c_2 k} \left(\intl_{D_j(w)} \abs{\dbar
\f(\tau)\rho(\tau)^{-l}}^2 d\nu(\tau)\right)^{1/2}\times\\
\times\left(\intl_{D_j(w)}  \frac{ \rho(\tau)^{2(l-1)} d\nu(\tau) }{ \abs{\scp{\partial{\rho(\tau)}}{\tau-z}}^n }\right)^{1/2}
= \sum\limits_{j>c_2 k} \intl_{d(w,z)<c_1 2^{-j}} b_j(w) m_j(w)
d\sigma(w)
\end{multline*}
Consider the integral~$m_j(w).$ Since $\tau\in D_j(w)$ then by
estimates from lemma~\ref{lm:QM_est2} $d(\tau,z) \asymp \rho(\tau) +
d(w,z)> 2^{-j}$ and
\begin{equation}
 m_j(w) = \left(\intl_{D_j(w)}  \frac{ \rho(\tau)^{2(l-1)} d\nu(\tau) }{ \abs{\V{\tau}{z}}^n }\right)^{1/2} \lesssim \frac{2^{-j(l-1)}}{2^{-jn}} 2^{-j} = 2^{jn-jl}.
\end{equation}
Thus
\begin{equation} \label{est_U}
2^{kl} U(z) \lesssim \sum\limits_{j>c_1 k} 2^{-(j-k)l}
2^{jn}\intl_{d(w,z)< c_2 2^{j} } b_j(w) d\sigma(w)
 \lesssim  \sum\limits_{j>c_1 k} 2^{-(j-k)l} Mb_j(z).
\end{equation}

Now estimate the value $W_1(z).$ Similarly to the previous we have
\begin{multline*}
W_1(z) \leq 2^{-kr} \suml_{j>k} \intl_{ d(w,z)\geq c_1 2^{-k} }
b_j(w) m_j^r(w) d\sigma(w)\\
\leq 2^{-kr} \suml_{j>k} \suml_{t=c_2}^k \intl_{ c_12^{-t}\leq d(w,z)\leq c_1 2^{-t+1} }
b_j(w) m_j^r(w) d\sigma(w),
\end{multline*}
where
\begin{equation*}
 m_j^r(w) =\left( \intl_{ D_j(w) } \frac{ \rho(\tau)^{ 2(l-1) } d\nu(\tau) }{ \abs{\V{\tau}{z}}^{2(n+r)} } \right)^{1/2}.
\end{equation*}
Applying the estimate $d(\tau,z) \asymp \rho(\tau) + d(w,z)\gtrsim
2^{-t},$ we obtain
\begin{equation*}
 m_j^r(w) \lesssim 2^{-jl+t(n+r)} .
\end{equation*}
Finally
\begin{equation*}
\suml_{t=c_2}^k \intl_{ d(w,z)\leq c_1 2^{-t+1} }
b_j(w) m_j^r(w) d\sigma(w) \lesssim \suml_{t=c_2}^k 2^{-jl+tr}Mb_j(z)\lesssim 2^{-jl+kr} Mb_j(z) 
\end{equation*}
and
\begin{equation} \label{est_W1}
 2^{kl} W_1(z) \lesssim \sum\limits_{j>k}  2^{-l(j-k)} Mb_j(z).
\end{equation}

Similarly, estimating the contribution of $W_2(z),$ we obtain

\begin{equation}
2^{kl}W_2(z) \lesssim 2^{-k(r-l)} \sum\limits_{j=0}^k \intl_\dom
b_j(w) m_j^r(w) d\sigma(w).
\end{equation}
Since $d(\tau,z) \gtrsim 2^{-j}+d(w,z)$ for $\tau\in\dom,\ \tau\in
D_j(z)$ then
$$m_j^r(w) \lesssim \frac{2^{-jl}}{(2^{-j} + d(w,z))^{n+r}}\leq \min\left(2^{j(n+r-l)}, 2^{-jl}d(w,z)^{-n-r}\right) .$$
Thus
\begin{multline*}
 \intl_\dom b_j(w) m_j^r(w) d\sigma(w) \lesssim \intl_{d(w,z)\leq
 2^{-j}} \frac{ 2^{-jl} }{2^{-j(n+r)}} b_j(w) d\sigma(w)\\ +\sum\limits_{t=1}^{j-1}\intl_{2^{-t-1}\leq d(w,z)\leq
 2^{-t}} \frac{ 2^{-jl} }{2^{-t(n+r)}} b_j(w) d\sigma(w)\\
 \lesssim \sum\limits_{t=1}^{j} 2^{-jl} 2^{tr} M b_j(z) \lesssim
 2^{-jl} 2^{j r} M b_j(z).
\end{multline*}
Choosing $r=2l,$ we have
\begin{equation} \label{est_W2}
 W_2(z) 2^{kl} \lesssim \sum\limits_{j=1}^k 2^{-(k-j)(r-l)} M
 b_j(z)\leq\sum\limits_{j=1}^k 2^{-(k-j)l} M
 b_j(z).
\end{equation}

Combining the estimates~(\ref{est_U},~\ref{est_W1},~\ref{est_W2}) we
finally obtain
$$\abs{f(z)-P_{2^k}(z)}2^{kl} \lesssim \sum\limits_{j=1}^k 2^{-(k-j)l} M
 b_j(z) + \sum\limits_{j>k}  2^{-(j-k) l} M b_j(z), $$
which similarly to \cite{D81} implies
$$ \SK \abs{f(z)-P_{2^k}(z)}^2  2^{2kl} \lesssim \SK (M b_k(z))^2. $$
Then, by Fefferman-Stein theorem
\begin{multline*}\intl_{\dom} d\sigma(z) \left(\SK \abs{f(z)-P_{2^k}(z)}^2 2^{2l k}
 \right)^{p/2}\leq \intl_{\dom} \left(\SK b^2_k(z)\right)^{p/2}d\sigma(z)\\ \leq \intl_{\dom} d\sigma(z) \left( \intl_{D^e(z,\eta,\eps)} \abs{\dbar \f(\xi) \rho(\xi)^{-l}}^2 d\nu(\xi) \right)^{p/2}<\infty.
 \end{multline*}
This  completes the proof of the theorem and it remains to prove
lemma~\ref{lm:Poly_Sobolev}. \qed
 \end{proof}

\begin{proof}[of the lemma~\ref{lm:Poly_Sobolev}]. Define
$g_k(z) := 2^{-kl}(P_{2^{k+1}}(z) - P_{2^k}(z)).$

Let $z\in\dom$ and $\tau\in S_k(z).$ Consider complex normal vector
$n(z) = \frac{\bar\partial\rho(z)}{ \abs{\bar\partial\rho(z)} }$ at~$z,$
complex tangent hyperplane $T_z=\{w\in\C^n:
\scp{\partial\rho(z)}{w-z}=0\}$ and complex plane
$T_{z,\tau}^{\perp}$, orthogonal to $T_z$ and containing the point
$\tau$
$$T_{z,\tau}^{\perp} := \{\tau+s n(z): s\in\C\}.$$
Projection of vector $\tau\in\C$ to $\dom\bigcap T_{z,\tau}^{\perp}$
we will denote as $\pi_z(\tau).$

Define $\O_{z,\tau} = \O\bigcap T_{z,\tau}^{\perp}$ and
$\gamma_{z,\tau}=\dom_{z,\tau}.$ There exist a conformal map\\
$\varphi_{z,\tau}:T_{z,\tau}^{\perp}\setminus\O_{z,\tau} \to
\C\setminus\{w\in\C: \abs{w}=1\}$ such that $\
\varphi_{z,\tau}(\infty)=\infty,\ \varphi_{z,\tau}'(\infty)>0, $ and
we could consider analytical in
$T_{z,\tau}^{\perp}\setminus\O_{z,\tau} $ function
$G_k(s):=\frac{g_k(s)}{\varphi_{z,\tau}^{2^{k+1}}(s)}.$

Applying to function $G_k$ Dyn'kin maximal estimate from~\cite{D77}
for domain $T_{z,\tau}^{\perp}~\setminus~\O(z,\tau)$ we obtain the
estimate
$$ \abs{G_k(\tau)}\lesssim  \frac{1}{\rho(\tau)} \intl_{s\in I_{z,\tau}}\abs{G_k(s)} \abs{ds} + \intl_{\dom_{z,\tau}\setminus I_{z,\tau}} \abs{G_k(s)}\frac{\rho(\tau)^m}{\abs{s-\pi_z(\tau)}^{m+1}} \abs{ds}, $$
where $I_{z,\tau} = \{s\in\gamma_{z,\tau}: \abs{s-\pi_z(\tau)}<
\dist{\tau}{\dom_{z,\tau}}/2\},$ and $m>0$ could be chosen arbitrary
large.

Note that $\abs{\varphi_{z,\tau}(s)}-1\asymp
\dist{s}{\dom_{z,\tau}}\asymp 2^{-k}, $ thus $\abs{g_k(s)}\asymp
\abs{G_k(s)}$ for $s\in D_k(z)\bigcap T_{z,\tau}^{\perp}.$ Hence,
\begin{equation} \label{lm:Poly_Sobolev_ineq1}
\abs{g_k(\tau)} \lesssim \suml_{j=1}^\infty 2^{-jm}
\frac{1}{2^j\rho(\tau)} \intl_{\substack{s\in\dom_{z,\tau}\\
\abs{s-\pi_z(\tau)}<2^j\rho(\tau)}} \abs{g_k(s)} \abs{ds}.
\end{equation}

Since the boundary of the domain $\O$ is $C^3$-smooth, we can assume
that the constant in this inequality~(\ref{lm:Poly_Sobolev_ineq1})
does not depend on $z\in\dom$ and $\tau\in\OO.$

Note that function $g_k(\tau+z-w)$ is holomorphic in $w\in T_z,$
then estimating the mean we obtain
\begin{multline}
 \abs{g_k(\tau)}\leq \frac{1}{\rho(\tau)^{n-1}} \intl_{\abs{w-z}<\sqrt{\rho(\tau)}} \abs{g_k(\tau+z-w)} d\mu_{2n-2}(w)\\
 \lesssim \suml_{j=1}^\infty 2^{-jm} \frac{1}{\rho(\tau)^{n-1}} \intl_{\abs{w-z}<\sqrt{\rho(\tau)}} \frac{d\mu_{2n-2}(w) }{2^j\rho(\tau)} \intl_{\substack{s\in\dom_{z,\tau}\\ \abs{s-\pi_z(\tau+z-w)}<2^j\rho(\tau)}} \abs{g_k(s)} \abs{ds}\\ \lesssim \suml_{j=1}^\infty
 2^{-j(m-n+1)} \intl_{B(z,2^j\rho(\tau))} \abs{g_k(w)}d\sigma(w),
\end{multline}
where $d\mu_{2n-2}$ is Lebesgue measure in $T_z$

Assume that $m>n-1,$ then $ \abs{g_k(\tau)} \lesssim M g_k(z),\
z\in\dom,\ \tau\in D_k(z).$ Finally,
\begin{multline*}
b_k(z) = \intl_{D_k(z)} \abs{\dbar \f(\tau)\rho(\tau)^{-l}}^2
d\nu(\tau) 
\lesssim \intl_{D_k(z)} \abs{
g_k(\tau)\rho(\tau)^{-l-1}}^2 d\nu(\tau)\\ \lesssim \left(M
a_k(z)\right)^{2} \intl_{D_k(z)} \frac{d\nu(\tau)}{\rho(\tau)^2}
\lesssim \left(Ma_k(z)\right)^{2}
\end{multline*}
and the lemma is proved.\hfill $\Box$
\end{proof}


\begin{appendix}
\renewcommand*{\thesection}{\Alph{section}}

\section{Area-integral inequality for external Kor\'{a}nyi region
}\label{Area_int} \numberwithin{theorem}{section}

Let $\O\subset\Cn$ be a strongly convex domain and $\eta,\eps>0$. For
function $g\in L^1(\dom)$ and $l\in\N$ we define a function
\begin{equation} \label{eq:area_int}
I_l(g,z) = \left(\ \intl_{D^e(z,\eta,\eps)} \abs{\ \intl_\dom \frac{ g(w)
dS(w)}{\V{\tau}{w}^{n+l}} }^2 d\nu_l(\tau) \right)^{1/2},
\end{equation}
where $dS(w) =\cn
\partial\rho(w)\wedge(\bar{\partial}\partial\rho(w))^{n-1}$ (see (\ref{CLF})) and  $d\nu_l(\tau) = \frac{d\mu_{2n}(\tau)}{\rho(\tau)^{n-2l-1}}.$

\begin{theorem}\label{thm:area_int}
Let $\O$ be strongly convex domain and $g\in L^p(\dom),\
1<p<\infty,$ Then
\begin{equation} \label{est:area_int}
    \intl_\dom I_l(g,z)^p d\sigma(z)  \lesssim \intl_\dom \abs{g(z)}^p
d\sigma(z).
\end{equation}

\end{theorem}
Note that in the one-variable case the integral~(\ref{eq:area_int})
gives the holomorphic function and the result of the theorem follows
from~\cite{D81}.

\begin{definition} Assume, that defining function~$\rho$ for strongly convex domain $\O$
has the following form near $0\in\dom$
\begin{equation} \label{eq:rhostandart}
 \rho(z) = 2\RE(z_n) + \suml_{j,k=1}^n A_{jk} z_j \bar{z}_k + O(\abs{z}^3)
\end{equation}
with positive definite form $A_{jk} z_j \bar{z}_k.$ We define a set
\begin{multline}
 D_0(\eta,\eps) = \{ \tau\in\CO : \abs{\tau_1}^2+\ldots+\abs{\tau_{n-1}}^2 < \eta \RE(\tau_n),\\ \abs{\IM(\tau_n)} <\eta\RE(\tau_n),\ \abs{\RE(\tau_n)}<\eps\}.
\end{multline}
\end{definition}

\begin{lemma}\label{lemma:rhostandart}
 Suppose, that $\rho$ has the form (\ref{eq:rhostandart}). There exist constants $c,\eps_0>0$ such that
    \begin{equation*} 
        D^e(0,\eta,\eps)\subset D_0(c\eta,c\eps),\ D_0(\eta,\eps)\subset D^e(0,c\eta,c\eps)\ \text{for}\ 0<\eta,\eps<\eps_0.
    \end{equation*}
\end{lemma}

\begin{proof} For the function
$\rho$ of the form~(\ref{eq:rhostandart}) the Kor\'{a}nyi sector (\ref{df:KoranyExt})
could be expressed as follows
\begin{multline*}
D^e(0,\eta,\eps) =  \{\tau\in \CO: \abs{\tau_1}^2+\ldots+\abs{\tau_{n-1}}^2\leq \eta\rho(\tau),\\ \abs{\IM(\tau_n)}\leq\eta\rho(\tau),\ \rho(\tau)<\eps\}
\end{multline*}
and
\begin{multline*}
\rho(\tau)\leq 2\RE(\tau_n) + c_0\left( \abs{\tau_1}^2+\ldots+\abs{\tau_{n-1}}^2 + \IM(\tau_n)^2 + \RE(\tau_n)^2\right)\\
\leq  (2+ c_0 \RE(\tau_n))\RE(\tau_n) + c_0(1+ \eta\rho(\tau))\eta\rho(\tau), \ \tau\in D^e(0,\eta,\eps) .
\end{multline*}
Thus for $\eta<\eta_0=\frac{1}{8c_0}$ we have $\rho(\tau)\leq c \RE(\tau_n).$

It is easy to see, that $|\tau|\to 0$ when $\rho(\tau)\to 0, \tau\in D^e(0,\eta,\eps).$  Then by convexity of $\O$
$$2\RE(\tau_n) =  \rho(\tau) - \suml_{j,k=1}^n A_{jk} \tau_j \bar{\tau}_k + O(\abs{\tau}^3)\leq \rho(\tau),\ \tau\in D^e(0,\eta,\eps_0) $$
for some $\eps_0\in(0,\eta_0).$

Finally $D^e(0,\eta,\eps)\subset D_0(c\eta,\eps)$ and analogously $D_0(\eta,\eps)\subset D^e(0,\eta,\eps)$ for $0<\eta,\eps<\eps_0.$ \qed

\end{proof}

\begin{theorem} \label{thm:rhostandart}
There exists such covering of the set 
$\overline{\O}_\eps\setminus\O_{-\eps}$ by open sets $\Gamma_j$ such
that for every $\xi\in\Gamma_j$ we can find a holomorphic change of
coordinates $\varphi_j(\xi,\cdot) : \Cn \to \Cn $ such that
\begin{enumerate}
    \item[{\rm{1.}}] The mapping $\varphi_j(\xi,\cdot)$
    transforms function $\rho$ to the type (\ref{eq:rhostandart})
    and could be expressed as follows
    \begin{equation}\label{thm:rhostandart:cond1}
    \varphi_j(\xi,z) = \Phi_j(\xi) (z-\xi) + (z-\xi)^\perp B_j(\xi) (z-\xi) e_n,
    \end{equation}
    where matrices  $\Phi_j(\xi), B_j(\xi)$ are $C^1$-smooth on
    $\Gamma_j,$ and $e_n=(0,\ldots,0,1).$
    \item[{\rm{2.}}] Let $\psi_j(\xi,\cdot)$ be an inverse map of $\varphi_j(\xi,\cdot),$ and let $J_j(\xi,\cdot)$ be a complex Jacobian of $\psi_j$.
    Then
    \begin{align} \label{thm:rhostandart:cond2}
     \sup\limits_{\tau\in\O_\eps\setminus\overline{\O}_\eps}
     \abs{J_j(\xi,\cdot) - J_j(\xi',\cdot)} &\lesssim \abs{\xi-\xi'},\\
     \sup\limits_{\tau\in\O_\eps\setminus\overline{\O}_\eps}
     \abs{\psi_j(\xi,\cdot) - \psi_j(\xi',\cdot)} &\lesssim \abs{\xi-\xi'}.\
    \end{align}
    Note that real Jacobian is then equal to $\abs{J_j(\xi,\cdot)}^2 = J_j(\xi,\cdot)\overline{J_j(\xi,\cdot)}.$
    \item[{\rm{3.}}] There exist constants $c,\eps_0>0$
    such that for $0<\eta,\eps<\eps_0$
    \begin{equation} \label{thm:rhostandart:cond3}
        \varphi_j(\xi,D^e(\xi,\eta,\eps))\subseteq D_0(c\eta,c\eps),\quad \psi_j(\xi,D_0(\eta,\eps)\subseteq
        D^e(\xi,c\eta,c\eps).
    \end{equation}
\end{enumerate}
\end{theorem}

\begin{proof}
Let $\xi\in\dom,$ by linear change of coordinates $z' =
(z-\xi)\Phi(\xi)$ we could obtain the following form for function
$\rho$ 
\begin{multline*}
\rho(z) = \rho(\xi+ \Phi^{-1}(\xi)z')\\ = 2\RE(z'_n) + \suml_{j,k=1}^n A^1_{jk}(\xi) z'_j\bar{z}'_k  + \RE\suml_{j,k=1}^n A_{jk}^2(\xi) z'_j z'_k  +O(\abs{z'}^3).
\end{multline*}
Setting $z''_n =z'_n +  A_{jk}^2 z'_j z'_k$ and $z''_j = z'_j,\
1\leq j\leq n-1,$ we have (see \cite{Ra86})
$$\rho(z'')= 2\RE(z'_n) + \suml_{j,k=1}^n A^1_{jk}(\xi) z''_j\bar{z}''_k   +O(\abs{z''}^3).  $$
Denote $B(\xi) = \Phi(\xi)^\perp A^2(\xi) \Phi(\xi),$ then
$$\varphi(\xi,z) = \Phi(\xi) (z-\xi) + (z-\xi)^\perp B(\xi)
(z-\xi) e_n.$$

We choose $\Gamma_j$ such that the matrix $\Phi(\xi)$ could be
defined on $\Gamma_j$ smoothly, this choice we denote as $\Phi_j,$
and the change corresponding to this matrix as $\varphi_j$
$$\varphi_j(\xi,z) = \Phi_j(\xi) (z-\xi) + (z-\xi)^\perp B_j(\xi)
(z-\xi)e_n.$$ Thus mappings $\varphi_j$ satisfy the first condition.
Easily, the second condition also holds.

The last condition (\ref{thm:rhostandart:cond3}) follows immediately from lemma~\ref{lemma:rhostandart}. This ends the
proof of the theorem. \qed
\end{proof}

Further we will assume, that the covering
$\overline{\O}_\eps\setminus\O_{-\eps} \subset
\bigcup\limits_{j=1}^N\Gamma_j$ and maps $\varphi_j, \psi_j$ are
chosen by the theorem~\ref{thm:rhostandart}. For covering
$\{\Gamma_j\}$ we consider a smooth decomposition of identity on
$\dom:$
$$\chi_j \in C^{\infty}(\Gamma_j),\ 0\leq\chi_j\leq 1,\ \supp{\chi_j}\subset\Gamma_j,\ \suml_{j=1}^N \chi_j(z) = 1,\ z\in\dom.$$

Fix parameters $0<\eps,\eta<\eps_0,$ denote $D_0 =D_0(\eta,\eps).$ Then by
(\ref{thm:rhostandart:cond3}) 
$$D^e(z)=\varphi_j(z,D^e(z,\eta/c,\eps/c))\subset D_0$$ and
\begin{multline} \label{eq:Luzin_decomposition}
I_l (g,z)^2\\ = \suml_{j=1}^N \chi_j(z)
\intl_{D^e(z)} \abs{\ \intl_\dom \frac{
g(w)J_j(z,\tau) dS(w)}{\V{\psi_j(z,\tau)}{w}^{n+l}} }^2
\frac{d\mu(\tau)}{\RE(\tau_n)^{n-2l+1}} \\
 \lesssim \suml_{j=1}^N  \intl_{D_0} \abs{\ \intl_\dom \frac{ g(w)\chi_j^{1/2}(z) J_j(z,\tau)
dS(w)}{\V{\psi_j(z,\tau)}{w}^{n+l}} }^2
\frac{d\mu(\tau)}{\RE(\tau_n)^{n-2l+1}}.
\end{multline}

\noindent We will consider the function
\begin{equation}
 K_j(z,w) (\tau) = \frac{\chi_j^{1/2}(z)J_j(z,\tau)}{\V{\psi_j(z,\tau)}{w}^{n+1}}
\end{equation}
as a map $\dom\times\dom\to \mathscr{L}(\C,L^2(D_0,d\nu_l)),$ such
that its values are operator of multiplication from $\C$ to
$L^2(D_0,d\nu_l),$ where
$d\nu_l(\tau)=\frac{d\mu(\tau)}{\IM(\tau_n)^{n-2l+1}}$ is a measure
on the region $D_0.$ Throughout the proof of the
theorem~\ref{thm:area_int} $j,l$ will be fixed integers and the norm
of function $F$ in the space $L^2(D_0,d\nu_l)$ will be denoted as~$\norm{F}.$

We will show that integral operator defined by kernel $K_j$ is
bounded on $L^p.$ To prove this we apply $T1$-theorem for
transformations with operator-valued kernels formulated by
Hyt\"{o}nen and Weis in \cite{HW05}, taking in account that in our
case concerned spaces are Hilbert. Some details of the proof are
similar to the proof of the boundedness of operator
Cauchy-Leray-Fantappi\`{e} $K_\O$ for lineally convex domains
introduced in \cite{R12}. Below we formulate the $T1$-theorem,
adapted to our context.

\begin{definition}
We say that the function $f\in C^\infty_0(\dom)$ is a normalized
bump-function, associated with the quasiball $B(w_0,r)$ if
$\supp{f}\subset B(z,r),$ $\abs{f}\leq 1,$ and
$$\abs{f(\xi)-f(z)}\leq \frac{d(\xi,z)^{\gamma}}{r^{\gamma}}.$$
The set of bump-functions associated with $B(w_0,r)$ is denoted as
$A(\gamma,w_0,r).$
\end{definition}

\begin{theorem} \label{thm:T1}
Let $K:\dom\times\dom\to \mathscr{L} (\C, L^2(D_0,d\nu_l))$ verify
the estimates
\begin{align}
    &\norm{K(z,w)} \lesssim \frac{1}{d(z,w)^n}; \label{KZ1}\\
    &\norm{K(z,w)-K(\xi,w)} \lesssim \frac{d(z,\xi)^\gamma}{d(z,w)^{n+\gamma}},\quad d(z,w)> C d(z,\xi); \label{KZ2}\\
    &\norm{K(z,w)-K(z,w')} \lesssim \frac{d(w,w')^\gamma}{d(z,w)^{n+\gamma}},\quad d(z,w)> C d(w,w').\label{KZ3}
\end{align}
Assume that operator $T:
\mathscr{S}(\dom)\to\mathscr{S}'(\dom,\mathscr{L} (\C,
L^2(D_0,d\nu_l)))$ with kernel $K$ verify the following conditions.
\begin{itemize}
 \item $T1,\ T'1\in\BMO(\dom,L^2(D_0,d\nu_l)),$ where $T'$ is
 formally adjoint operator.
 \item Operator $T$ satisfies the weak boundedness property, that is
 for every pair of normalized bump-functions $f,g\in A(\gamma,w_0,r)$ we have $$\norm{\scp{g}{Tf}}\leq C
 r^{-n}.$$
\end{itemize}
Then $T\in\mathscr{L} (L^p(\dom), L^p(\dom,L^2(D_0,d\nu_l))$ for
every $p\in (1,\infty).$
\end{theorem}
\medskip

In the following three lemmas we will prove that kernels $K_j$ and
corresponding operators $T_j$ satisfy the conditions of the
$T1$-theorem.

\begin{lemma} \label{lm:T1_1} The kernel $K_j$ verify estimates (\ref{KZ1}-\ref{KZ3}).

\end{lemma}
\begin{proof} By lemma \ref{lm:QM_est2} we have 
$\abs{\V{\tau}{w}}\asymp\rho(\tau) + \abs{\V{z}{w}},$ $z,w\in\dom,\
\tau\in D^e(z,c\eta,c\eps). $
Thus
\begin{multline*}
\norm{ K_j(z,w)}^2 = \intl_{D_0} \abs{K_j(z,w)(\tau)}^2 d\nu_l(\tau)
\lesssim\intl_{D^e(z,c\eta,c\eps)}
\frac{d\nu_l(\tau)}{ \abs{ \V{\tau}{w} }^{2n+2l} } \\
 \lesssim\intl_{D^e(z,c\eta,c\eps)}\frac{1}{( \rho(\tau) + \abs{\V{z}{w}} )^{2n+2l}}
\frac{d\mu(\tau)}{\rho(\tau)^{n-2l+1}}\\
\lesssim \intl_0^\infty\frac{t^{2l-1} dt}{(t+\abs{\V{z}{w}})^{2n+2l}} \lesssim
\frac{1}{\abs{\V{z}{w}}^{2n}}\lesssim \frac{1}{d(z,w)^{2n}}.
 \end{multline*}

Similarly,
\begin{multline*}\norm{K_j(z,w) - K_j(z,w')}^2\\
\lesssim\intl_{D^e(z,c\eta,c\eps)} \abs{\frac{1}{\V{\tau}{w}^{n+l}} -
\frac{1}{\V{\tau}{w'}^{n+l}}}^2 d\nu_l(\tau)
\end{multline*}

%

\noindent Denote $\hat{\tau} =\pr_{\dom}(\tau),$ then
\begin{multline*} \abs{\V{\tau}{w}} \lesssim \rho(\tau) +
\abs{\V{\hat{\tau}}{w}}\\ \lesssim \rho(\tau) +  \abs{\V{z}{w}} +
\abs{\V{\hat{\tau}}{z}} \lesssim \rho(\tau) + \abs{\V{z}{w}},
\end{multline*}
which combined with lemma \ref{lm:QM_est2} and condition $$d(w,w')=\abs{\V{w}{w'}}<
C \abs{\V{z}{w}}=C d(z,w)$$ implies
\begin{multline*}
\abs{\V{\tau}{w}} \asymp \rho(\tau) +  \abs{\V{z}{w}} \asymp \rho(\tau) +  \abs{ \V{z}{w'} }\\ \asymp \abs{\V{\tau}{w'}}.  
\end{multline*}
Next, we have 
\begin{multline*}
 \abs{ \V{\tau}{w'}-\V{\tau}{w} } =  \abs{ \scp{\partial\rho(\tau)}{\hat{\tau}-w}-\scp{\partial\rho(\tau)}{\hat{\tau}-w'} }\\
  \leq \abs{ \scp{\partial\rho(\tau)-\partial\rho(\hat{\tau})}{w-w'} } +
  \abs{ \scp{\partial\rho(\hat{\tau})}{\hat{\tau}-w}-\scp{\partial\rho(\hat{\tau})}{\hat{\tau}-w'} }\\
  \lesssim \rho(\tau) \abs{\V{w}{w'}}^{1/2} +
  \abs{\V{\hat{\tau}}{w}}^{1/2}\abs{\V{w}{w'}}^{1/2}\\ \lesssim   \abs{\V{\tau}{w}}^{1/2} \abs{\V{w}{w'}}^{1/2}
\end{multline*}
Hence,
\begin{multline*}\norm{K_j(z,w) - K_j(z,w')}^2\lesssim
\intl_{D^e(z,c\eta,c\eps)} \frac{ \abs{\V{w}{w'}} }{ \abs{\V{\tau}{w}}^{2n+2l+1} } d\nu_l(\tau)\\
\lesssim  \intl_0^\infty \frac{\abs{\V{w}{w'}} t^{2l-1}
dt}{(t+\abs{\V{z}{w}})^{2n+2l+1}} \lesssim \frac{ \abs{\V{w}{w'}} }{
\abs{ \V{z}{w} }^{2n+1} }=\frac{d(w,w')}{d(z,w)^{n+1}}.
\end{multline*}

The last inequality (\ref{KZ3}) is a bit harder to prove. 

Let $z,\xi,w\in\dom,\ Cd(z,\xi)<d(z,w),$ and estimate the value
$$A = \abs{\V{\psi_j(z,\tau)}{w} - \V{\psi_j(\xi,\tau)}{w}  }. $$
Denote $\tau_z = \psi_j(z,\tau),\ \tau_\xi =\psi_j(\xi,\tau),$ then by (\ref{thm:rhostandart:cond1})
\begin{multline*}
\tau = \Phi(z)(\tau_z-z)+i (\tau_z-z)^{T}B(z)(\tau_z-z)e_n\\ = \Phi(\xi)(\tau_\xi-\xi)+i
(\tau_\xi-\xi)^{T}B(\xi)(\tau_\xi-\xi)e_n,
\end{multline*} whence denoting $\Psi(z) = \Phi(z)^{-1}$ and introducing $L(z,\xi,\tau)$ we obtain
\begin{align*}
\tau_z &= z+ \Psi(z)\tau -  (\tau_z-z)^{T}B(z)(\tau_z-z)\Psi(z)e_n ,\\
\tau_\xi &= \xi+ \Psi(\xi)\tau - (\tau_\xi-\xi)^{T}B(\xi)(\tau_\xi-\xi)\Psi(\xi)e_n ,\\
\tau_z-\tau_\xi &= z-\xi + (\Psi(z)-\Psi(\xi))\tau + L(z,\xi,\tau)e_n .
\end{align*}
Note, that norms of matrices $\norm{\Psi(\xi)}$ are bounded, thus
\begin{multline*}
 \abs{L(z,\xi,\tau)} \leq \abs{(\tau_z-z)^{T}B(z)(\tau_z-z)(\Psi(z)-\Psi(\xi))}\\ + \abs{(\tau_z-z)^{T}B(z)(\tau_z-z) -
 (\tau_\xi-\xi)^{T} B(\xi)(\tau_\xi-\xi)} \norm{\Psi(\xi)}\\
 \lesssim \abs{z-\xi}\abs{\tau_z-z}^2 + \abs{(\tau_z-z-\tau_\xi+\xi)^T
 B(z)(\tau_z-z)}\\
 + \abs{(\tau_\xi-\xi)^T B(z)(\tau_z-z) - (\tau_\xi-\xi)^T B(\xi)(\tau_\xi-\xi)}\\
  \lesssim \abs{z-\xi}\abs{\tau_z-z}^2 + \abs{z-\xi}\abs{\tau} + \abs{ ((\Psi(z)-\Psi(\xi))\tau+L(z,\xi,\tau)e_n)^T
  B(z)(\tau_z-z)}\\ + \abs{(\tau_\xi-\xi)^T(B(z)-B(\xi))(\tau_z-z)} + \abs{(\tau_\xi-\xi)^T B(\xi)
  (\tau_z-z-\tau_\xi-\xi)}\\
  \lesssim \abs{z-\xi}\abs{\tau_z-z}^2 + \abs{z-\xi}\abs{\tau} +\abs{\tau}
  \abs{L(z,\xi,\tau)}
  +\abs{z-\xi}\abs{\tau}^2 + \abs{\tau} L(z,\xi,\tau).
\end{multline*}

Choosing $\eps>0$ small enough we get 
$\abs{\tau}\leq\eta\abs{\IM(\tau_n)}+(1+\eta)\abs{\IM(\tau_n)}\leq3{\eps}$
and
$\abs{L(z,\xi,\tau)} \lesssim d(z,\xi)^{1/2}\abs{\tau},$ for $\tau\in D_0=D_0(\eta,\eps).$ Hence,
\begin{multline*}
 A\leq
 \abs{\scp{\partial\rho(\tau_z)-\partial\rho(\tau_\xi)}{\tau_z-w}} +
 \abs{\scp{\partial\rho(\tau_\xi)}{\tau_z-w}}\\
 \lesssim \abs{\tau_z-\tau_\xi}(\rho(\tau_z)+d(z,w)^{1/2}) +
 \abs{\scp{\partial\rho(\tau_z)-\partial\rho(\tau_\xi)}{z-\xi}} +
 \abs{\V{z}{\xi}}\\
 + \abs{\scp{\partial\rho(\tau_\xi)}{(\Psi(z)-\Psi(\xi))\tau}} + \abs{\scp{\partial\rho(\tau_\xi)}{L(z,\xi,\tau)}} \lesssim d(z,\xi)^{1/2}d(\tau_z,w) +\\
 \abs{\tau_z-\xi}\abs{z-\xi}+d(z,\xi)+\abs{z-\xi}\abs{\tau}+\abs{L(z,\xi,\tau)} \lesssim d(z,\xi) + d(z,\xi)^{1/2}d(z,w)^{1/2}\\
 \lesssim  d(z,\xi)^{1/2}d(z,w)^{1/2}
\end{multline*}
Combining this estimate with inequality
$\abs{\V{\tau_z}{w}}\asymp\abs{\V{\tau_\xi}{w}}$ we obtain
\begin{multline*}
\norm{K_j(z,w) - K_j(\xi,w)}^2 \lesssim
 \intl_{D^e(z,c\eta,c\eps)}
\frac{\abs{\chi_j(z)^{1/2}-\chi_j(\xi)^{1/2}}^2}{\abs{\V{\tau}{w}}^{2n+2l}} \frac{d\mu(\tau)}{\rho(\tau)^{n-2l+1}}\\
+ \chi_j(\xi) \intl_{D_0}
\frac{ \abs{\V{z}{\xi}}\abs{\V{z}{w}} }{ \abs{\V{\tau_z}{w}}^{2n+4} }\frac{d\mu(\tau)}{\RE(\tau_n)^{n-2l+1}}\\
\lesssim \frac{ \abs{\V{z}{\xi}} }{ \abs{\V{z}{w}}^{2n} } + \frac{
\abs{\V{z}{\xi}} }{ \abs{\V{z}{w}}^{2n+1} } \lesssim \frac{
\abs{\V{z}{\xi}} }{ \abs{\V{z}{w}}^{2n+1} }\\ 
\lesssim \frac{d(z,\xi)}{d(z,w)^{2n+1}}.
\end{multline*}\qed
\end{proof}

\begin{lemma} \label{lm:T1_2} $T_j(1) =0$ and $\norm{T_j'(1)} \lesssim 1.$
\end{lemma}

\begin{proof}
Introduce the notation $\tau_z = \psi_j(z,\tau).$ The function
$\V{\tau_z}{w}$ is holomorphic in $\O$ with respect to $w,$ then the
form $\V{\tau_z}{w}^{-n-l}dS(w)$ is closed in $\O$ and
$$T_j(1)(\tau) = \chi_j(z)^{1/2} J_j(z,\tau) \intl_\dom \frac{dS(w)}{\V{\tau_z}{w}^{n+l}} = 0. $$
It remains to estimate the value of formally-adjoint operator $T'$
on $f\equiv1$.
\begin{multline*}
 T'_j(1)(w)(\tau)= \intl_\dom \frac{\chi_j(z)^{1/2}
J_j(z,\tau) dS(z)}{\V{\tau_z}{w}^{n+l}}\\ = \intl_\dom
\frac{\chi_j(z)^{1/2} J_j(z,\tau)
(dS(z)-dS(\tau_z))}{\V{\tau_z}{w}^{n+l}} + \intl_\dom
\frac{\chi_j(z)^{1/2} J_j(z,\tau) dS(\tau_z)}{\V{\tau_z}{w}^{n+l}} =
L_1+L_2.
\end{multline*}
Note that $ \abs{z-\tau_z}\lesssim \RE(\tau_n),$ therefore
$\abs{dS(z)-dS(\psi(z,\tau))} \lesssim \RE(\tau_n)d\sigma(z)$ and
\begin{multline*}
\abs{L_1} \lesssim \intl_\dom \frac{\RE(\tau_n)d\sigma(z)}{
\abs{\V{\tau_z}{w}}^{n+l} } \lesssim
\frac{\RE(\tau_n)d\sigma(z)}{(\RE(\tau_n)+ \abs{\V{z}{w}})^{n+l} }\\
\lesssim \intl_0^\infty \frac{\RE(\tau_n)
v^{n-1}dv}{(\RE(\tau_n)+v)^{n+l}}\lesssim\frac{1}{\RE(\tau_n)^{l-1}}.
\end{multline*}
Thus we get
\begin{equation} \label{T'1_I1}
\intl_{D_0} \abs{L_1}^2 d\nu_l(\tau) \lesssim \intl_{D_0}
\frac{1}{\RE(\tau_n)^{2l-2}}\frac{d\mu(\tau)}{\RE(\tau_n)^{n-2l+1}}
\lesssim\intl_0^{\eps} \frac{t^ndt}{t^{n-1}} \lesssim 1
\end{equation}
To estimate $L_2$ we recall that $d_\xi
\frac{dS(\xi)}{\V{\xi}{z}^{n}} =0,\ z\in\dom,\ \xi\in\CO,$  and
consequently
\begin{multline*}
d \frac{dS(\xi)}{\V{\xi}{z}^{n+l}} =
\frac{(\dbar\partial\rho(\xi))^n}{\V{\xi}{z}^{n+l}}\\
 - (n+l)\frac{(\dbar_\xi\left(\V{\xi}{z}\right)\wedge\dbar\partial\rho(\xi))^{n-1}}{\V{\xi}{z}^{n+l}}
= -\frac{l}{n} \frac{dV(\xi)}{\V{\xi}{z}^{n+l}}.
\end{multline*}
By Stokes' theorem we obtain
\begin{multline*}
L_2= \intl_\dom \frac{\chi_j(z)^{1/2} J_j(z,\tau)
dS(\tau_z)}{\V{\tau_z}{w}^{n+l}}\\ = \intl_{\O_{\eps_1}\setminus\O}
\frac{\dbar_z\left(\chi_j(z)^{1/2} J_j(z,\tau)\right)\wedge
dS(\tau_z)}{\V{\tau_z}{w}^{n+l}} -\frac{l}{n}
\intl_{\O_{\eps_1}\setminus\O} \frac{\chi_j(z)^{1/2} J_j(z,\tau)
dV(\tau_z)}{\V{\tau_z}{z}^{n+l}}
\end{multline*}
Analogously to lemma \ref{lm:QM_est2} we have $\abs{\V{\tau_z}{w}}
\asymp \IM(\tau_n) + \rho(z) + \abs{\V{\hat{z}}{w}},$ where $\hat{z}=\pr_{\dom}(z).$ Hence,
\begin{multline*}
 \abs{L_2}\lesssim \intl_{\O_{\eps_1}\setminus\O}
\frac{d\mu(z)}{ \abs{\V{\tau_z}{w}}^{n+l} }\\
\lesssim\intl_0^{\eps}
dt\intl_{\dom_t}\frac{d\sigma_t}{(t+\IM(\tau_n)+ \abs{\V{\hat{z}}{w}} )^{n+l}}\\
\lesssim \intl_0^{\eps} dt \intl_0^\infty
\frac{v^{n-1}dv}{(t+\RE(\tau_n)+v)^{n+l}}\lesssim \intl_0^{\eps}
\frac{dt}{(t+\RE(\tau_n))^l}\\ \lesssim (\RE(\tau_n))^{1-l}
\ln{\left(1+\frac{1}{\RE(\tau_n)}\right)},
\end{multline*}
and
\begin{multline*}
\intl_{D_0} \abs{L_2}^2d\nu_l(\tau) \lesssim \intl_{D_0} (\RE(\tau_n))^{2-2l}\ln^2\left(1+\frac{1}{\RE(\tau_n)}\right)d\nu_l(\tau)\\ \lesssim\intl_0^{\eps} \ln^2{\left(1+\frac{1}{s}\right)} s ds\lesssim 1,
\end{multline*}
which with the estimate (\ref{T'1_I1}) completes the proof of the
lemma. \qed

\end{proof}

\begin{lemma} \label{lm:T1_3} Operator $T_j$ is weakly bounded.
\end{lemma}
\begin{proof}
 Let $f,g\in A(\frac{1}{2},w_0,r),$ denote again
$\tau_z=\psi_j(z,\tau),$ then
$$\norm{\scp{g}{Tf}}^2 \lesssim \intl_{D_0} d\nu_l(\tau) \left( \intl_{B(w_0,r)} \abs{g(z)} dS(z) \abs{\intl_{B(w_0,r)}\frac{f(w) dS(w)}{\V{\tau_z}{w}^{n+1}}
} \right)^2.$$

Denote $t:=\inf\limits_{w\in\dom} \abs{\V{\tau_z}{w}} $ and
introduce the set
$$W(z,\tau,r) :=\left\{ w\in\dom: \abs{\V{\tau_z}{w}}<t+r\right\}.$$
Note that $B(w_0,r)\subset W(z,\tau,cr) \subset B(z,c^2r)$ for some
$c>0,$ therefore,
\begin{multline*}
\abs{\intl_{B(w_0,r)}\frac{f(w) dS(w)}{\V{\tau_z}{w}^{n+l}} }\\
=\abs{\intl_{W(z,\tau,cr)}\frac{f(w)
dS(w)}{\V{\tau_z}{w}^{n+l}} } \lesssim
\intl_{W(z,\tau,cr)}\frac{\abs{f(z)-f(w)} dS(w)}{
\abs{\V{\tau_z}{w}}^{n+l} }\\ + \abs{f(z)} \abs{\intl_{\dom\setminus
W(z,\tau,cr)}\frac{dS(w)}{\V{\tau_z}{w}^{n+l}} } = L_1(z,\tau) +
\abs{f(z)} L_2(z,\tau).
\end{multline*}
It follows from the estimate $|f(z)-f(w)|\leq \sqrt{v(w,z)/r}$ that
\begin{multline*}
L_1(z,\tau)\lesssim \frac{1}{\sqrt{r}}\intl_{B(z,c^2r)}
\frac{v(w,z)^{1/2}}{(\RE(\tau_n) + v(w,z))^{n+l}}\lesssim
\frac{1}{\sqrt{r}} \intl_0^{c^2r} \frac{t^{n-1/2}dt}{(\RE(\tau_n) + t)^{n+l}}\\
\lesssim \frac{1}{\sqrt{r}} \intl_0^{c^2r}\frac{dt}{(\RE(\tau_n) +
t)^{l+1/2}} \lesssim
\frac{1}{\sqrt{r}}\left(\frac{1}{\RE(\tau_n)^{l-1/2}}-
\frac{1}{(\RE(\tau_n)+r)^{l-1/2}}\right)\\ = \frac{1}{\sqrt{r}}
\frac{(\RE(\tau_n)+r)^{l-1/2} - r^{l-1/2}}{\RE(\tau_n)^{l-1/2}
(\RE(\tau_n)+r)^{l-1/2} } \lesssim \frac{1}{\sqrt{r}}
\frac{(\RE(\tau_n)+r)^{2l-1} - r^{2l-1}}{
\IM(\tau_n)^{l-1/2}(\RE(\tau_n)+r)^{2l-1}}\\ \lesssim
\frac{1}{\sqrt{r}} \frac{r \RE(\tau_n)^{2l-2} +
r^{2l-1}}{\RE(\tau_n)^{l-1/2}(\RE(\tau_n)+r)^{2l-1}}.
\end{multline*}
Estimating the $L^2(D_0,d\nu_l)-$norm of the function $L_1(z,\tau),$
we obtain
\begin{multline} \label{est:I1}
\intl_{D_0(\tau)} L_1(z,\tau)^2 d\nu_l(\tau)\\ \lesssim
\intl_{D_0(\tau)} \left(  \frac{r
\RE(\tau_n)^{2l-3}}{(\RE(\tau_n)+r)^{4l-2}} +
\frac{r^{4l-3}}{\RE(\tau_n)^{2l-1}(\RE(\tau_n)+r)^{4l-2}} \right)
\frac{d\mu(\tau)}{\RE(\tau_n)^{n-2l+1}}\\
\lesssim r \intl_0^\infty \frac{s^{4l-4}}{(s+r)^{4l-2}} ds +
r^{4l-3} \intl_0^\infty \frac{ds}{(s+r)^{4l-2}}\lesssim 1
\end{multline}
To estimate the second summand $L_2$ we apply the Stokes theorem to the
domain
$$\left\{ w\in\O:
\abs{\V{\tau_z}{w}} > t+c r\right\}$$ and to the closed in this
domain form $\frac{dS(w)}{\V{\tau_z}{w}^{n+l}}$
\begin{multline*}
\intl_{\dom\setminus W(z,\tau,cr)}\frac{dS(w)}{\V{\tau_z}{w}^{n+l}}
= - \intl_{\substack{w\in\O\\ \abs{v(\tau_z,w)} = t+cr}
}\frac{dS(w)}{\V{\tau_z}{w}^{n+l}}\\ = -\frac{1}{(t+cr)^{2n+2l}}
\intl_{\substack{w\in\O\\ \abs{v(\tau_z,w)} =
t+cr}}\overline{\V{\tau_z}{w}}^{n+l} dS(w).
\end{multline*}
Applying Stokes' theorem again, now to the domain $$\left\{ w\in\O:
\abs{\V{\tau_z}{w}}<t+cr\right\},$$ we obtain
\begin{multline*}
L_3:=\intl_{\substack{w\in\O\\
\abs{v(\tau_z,w)}=t+cr}}\overline{\V{\tau_z}{w}}^{n+l} dS(w)\\ =
-\intl_{\substack{w\in\dom\\ \abs{v(\tau_z,w)}<t+cr}}\overline{\V{\tau_z}{w}}^{n+l} dS(w)\\
+ \intl_{\substack{w\in\O\\
\abs{v(\tau_z,w)}<t+cr}}
\dbar_w\left(\overline{\V{\tau_z}{w}}^{n+l}\right)\wedge dS(w)\\ +
\intl_{\substack{w\in\O\\
\abs{v(\tau_z,w)}<t+cr}} \overline{\V{\tau_z}{w}}^{n+l} dV(w).
\end{multline*}
Since $ \abs{\
\dbar_w\left(\overline{\V{\tau_z}{w}}^{n+l}\right)\wedge dS(w)}
\lesssim \abs{\V{\tau_z}{w}}^{n+l-1}$ we get
$$
\abs{L_3} \lesssim \intl_t^{t+cr} (s^{n+l}s^{n-1} + s^{n+l}s^{n} +
s^{n+l-1}s^n) ds\lesssim \intl_t^{t+cr} s^{2n+l-1} ds \lesssim r
(t+r)^{2n+l-1}.
$$
Note that $t\asymp \rho(\tau_z)\asymp \IM(\tau_n)$ and consequently
\begin{multline} \label{est:I2}
\intl_{D_0} L_2(z,\tau)^2 d\nu_l(\tau) \lesssim
\intl_{D_0}  \left(\frac{r
(\RE(\tau_n)+r)^{2n+l-1}}{(\RE(\tau_n)+r)^{2n+2l}}\right)^2
d\nu_l(\tau)\\ \lesssim \intl_0^\infty \frac{r^2}{(t+r)^{2l+2}}
\frac{t^n dt}{t^{n-2l+1}} = r^2\intl_0^\infty
\frac{t^{2l-1}}{(t+r)^{2l+2}} \frac{t^n dt}{t^{n-2l+1}} \lesssim r^2
\intl_0^\infty \frac{dt}{(r+t)^3} \lesssim 1.
\end{multline}

Summarizing estimates (\ref{est:I1}, \ref{est:I2}) and condition
$\abs{f(z)}\leq 1,\ z\in\dom,$ we obtain
\begin{multline*}
\norm{\scp{g}{Tf}}^2 \leq \intl_{D_0} d\nu_l(\tau) \left(
\intl_{B(w_0,r)} \abs{g(z)} (L_1(z,\tau) + L_2(z,\tau)|f(z)|) dS(z)\right)^2 \\
\lesssim \norm{g}_{L^1(\dom)}^2
\sup\limits_{z\in\dom}\intl_{D_0}\left(L_1(z,\tau)^2 + L_2(z,\tau)^2\right)
d\nu_l(\tau)\\ \lesssim \norm{g}_{L^1(\dom)}^2\lesssim
\abs{B(w_0,r)}^2.
\end{multline*}
The last estimate implies weak boundedness of operator $T$ and
completes the proof of the lemma. \qed
\end{proof}

\begin{proof}[of the theorem \ref{thm:area_int}] Since operators $T_j$ with kernels $K_j$
verify the conditions of $T1$-theorem, we have $T_j\in\mathscr{L}
(L^p(\dom), L^p(\dom,L^2(D_0,d\nu_l))$ and
\begin{multline*}
 \suml_{j=1}^N \intl_\dom \norm{T_j g(z)}^p dS(z)\\ =
\suml_{j=1}^N \intl_\dom dS(z) \left(\ \intl_{D_0} \abs{\ \intl_\dom
\frac{ g(w)\chi_j^{1/2}(z) J_j(z,\tau)
dS(w)}{\V{\psi_j(z,\tau)}{w}^{n+1}} }^2
\frac{d\mu(\tau)}{\RE(\tau_n)^{n-1}}\right)^p\\ \lesssim
\norm{g}^p_{L^p(\dom)}.
\end{multline*}
Thus by decomposition (\ref{eq:Luzin_decomposition}) $\intl_\dom
I_l(g,z)^p\ d\sigma(z) \lesssim \intl_\dom \abs{g(z)}^p\ d\sigma(z),$
which proves the theorem. \qed
\end{proof}

\end{appendix}

\end{document}